\RequirePackage[l2tabu,orthodox]{nag}
\documentclass[11pt]{amsart}
\usepackage{mathtools,amsthm,amssymb,amsbsy,amstext,amsopn}
\usepackage{aliascnt}
\usepackage{xcolor}
\usepackage{graphicx}
\usepackage{tikz-cd}
\usepackage{microtype}
\usepackage[centering,margin=1.2in]{geometry}
\usepackage[colorlinks,
            linkcolor=black!50!red,
            citecolor=blue,
            pdfpagemode=None]{hyperref}
\usepackage{tikz}
\usetikzlibrary{calc,arrows,decorations.pathmorphing}
\numberwithin{equation}{section}

\theoremstyle{plain}
\newaliascnt{theorem}{equation}
\newtheorem{thm}[theorem]{Theorem}
\aliascntresetthe{theorem}

\newtheorem*{thm*}{Theorem}

\newaliascnt{prop}{equation}
\newtheorem{prop}[prop]{Proposition}
\aliascntresetthe{prop}

\newaliascnt{lemma}{equation}
\newtheorem{lemma}[lemma]{Lemma}
\aliascntresetthe{lemma}

\newaliascnt{corollary}{equation}

\aliascntresetthe{corollary}

\theoremstyle{definition}

\newaliascnt{remark}{equation}
\newtheorem{remarkthm}[remark]{Remark}
\newenvironment{remark}[1][]{\begin{remarkthm}[#1]\pushQED{\qed}}{\popQED \end{remarkthm}}
\aliascntresetthe{remark}

\newaliascnt{defn}{equation}
\newtheorem{defnthm}[defn]{Definition}
\newenvironment{defn}[1][]{\begin{defnthm}[#1]\pushQED{\qed}}{\popQED \end{defnthm}}
\aliascntresetthe{defn}

\newaliascnt{example}{equation}
\newtheorem{examplethm}[example]{Example}
\newenvironment{example}[1][]{\begin{examplethm}[#1]\pushQED{\qed}}{\popQED \end{examplethm}}
\aliascntresetthe{example}

\usepackage{cleveref}

\crefname{theorem}{Theorem}{Theorems}
\crefname{prop}{Proposition}{Propositions}
\crefname{section}{Section}{Sections}
\crefname{subsection}{Section}{Sections}
\crefname{defn}{Definition}{Definitions}
\crefname{remark}{Remark}{Remarks}
\crefname{lemma}{Lemma}{Lemmas}
\crefname{corollary}{Corollary}{Corollaries}
\crefname{example}{Example}{Examples}

\newcommand{\mb}[1]{\mathbf{#1}}

\DeclareMathOperator{\Ad}{Ad}

\DeclareMathOperator{\Hom}{Hom}

\DeclareMathOperator{\corank}{corank}

\newcommand{\wt}{\widetilde}
\newcommand{\wh}{\widehat}
\newcommand{\mr}{\mathring}

\newcommand{\ot}{\otimes}

\newcommand{\comments}[1]{}
\newcommand{\ol}[1]{\overline{#1}}
\newcommand{\dol}[1]{\overline{\overline{#1}}}

\def\al{\alpha}

\def\de{\delta}
\def\De{\Delta}

\def\io{\iota}
\def\ka{\kappa}

\def\th{\theta}
\def\si{\sigma}
\def\Si{\Sigma}

\def\om{\omega}

\def\ep{\epsilon}

\def\vphi{\varphi}

\def\CA{{\mathcal A}}

\def\CN{{\mathcal N}}

\def\CX{{\mathcal X}}

\def\BC{{\mathbb C}}

\def\BQ{{\mathbb Q}}

\def\BZ{{\mathbb Z}}

\newcommand{\fg}{\mathfrak g}

\newcommand{\xqedhere}[2]{
  \rlap{\hbox to#1{\hfil\llap{\ensuremath{#2}}}}}

\newcommand\isom{\overset{\sim}{\to}}
\newcommand\mono{\hookrightarrow}
\newcommand\into\mono
\newcommand\epi{\twoheadrightarrow}
\newcommand\onto\epi

\newcommand\<\langle
\renewcommand\>\rangle
\newcommand\sminus{\smallsetminus}

\begin{document}

\title{Q-Systems, Factorization Dynamics, and the Twist Automorphism}
\author{Harold Williams}
\address{Harold Williams\newline
University of California, Berkeley\newline
Department of Mathematics\newline
Berkeley CA 94720\newline
USA}
\email{harold@math.berkeley.edu}

\begin{abstract}
We provide a concrete realization of the cluster algebras associated with $Q$-systems as amalgamations of cluster structures on double Bruhat cells in simple algebraic groups.  For nonsimply-laced groups, this provides a cluster-algebraic formulation of $Q$-systems of twisted type.  It also yields a uniform proof of the discrete integrability of these $Q$-systems by identifying them with the dynamics of factorization mappings on quotients of double Bruhat cells.  On the double Bruhat cell itself, we find these dynamics are closely related to those of the Fomin-Zelevinsky twist map.  This leads to an explicit formula expressing twisted cluster variables as Laurent monomials in the untwisted cluster variables obtained from the corresponding mutation sequence.  This holds for Coxeter double Bruhat cells in any symmetrizable Kac-Moody group, and we show that in affine type the analogous factorization mapping is also integrable.



\end{abstract}

\maketitle

\section{Introduction}

$Q$-systems are nonlinear recurrence relations associated with affine Dynkin diagrams, arising in the Bethe ansatz and the representation theory of Yangians and quantum loop algebras \cite{Kirillov1990,Nakajima2000,Hernandeza,Hernandez2010}.  There is by now a large literature related to them and their relatives (see \cite[Section 13]{Kuniba2011} for a survey), and in particular it was discovered in \cite{Kedem2013,DiFrancesco2013} that they may be realized as sequences of cluster transformations in certain cluster algebras.  Our aim is to provide concrete realizations of these cluster algebras in terms of double Bruhat cells and their amalgamations.  The relevant sequences of cluster transformations are then identified with factorization mappings on quotients of double Bruhat cells, leading to their discrete integrability.  Moreover, these sequences provide an alternate description of the Fomin-Zelevinsky twist automorphism in terms of cluster transformations, yielding explicit formulas relating twisted and untwisted cluster variables.  

Informally, a cluster structure on a variety consists of an infinite family of toric charts with transition functions of a specific form, called cluster transformations \cite{Fomin2001}.  A basic class of examples are those on the double Bruhat cells $G^{u,v}$ of a simple algebraic group \cite{Berenstein2005}.  When $u=v$ is a Coxeter element $c$, we identify a $Q$-system whose associated cluster structure is related to that of $G^{c,c}$ by amalgamation, a type of quotient operation on cluster structures.

\begin{thm*}\emph{(\ref{prop:coxconjquot}, \ref{prop:qsystemsasclusters})}
The conjugation quotient $G^{c,c}/H$ has a natural cluster structure obtained from that of $G^{c,c}$ by amalgamation, whose exchange matrix is of the form
\[
B_C \coloneqq \begin{pmatrix} 0 & C^t \\ -C^t & 0 \end{pmatrix}.
\]
Up to normalization, there is a $Q$-system which can be realized by exchange relations in the corresponding cluster algebra; its type is the affinization of that of $G$ when this is simply-laced, otherwise it is of a twisted type related to that of $G$ by folding.
\end{thm*}

When $G$ is of type $A_n$ this reformulates a result of \cite{Gekhtman2011}, and our use of amalgamation to construct cluster structures on adjoint quotients generalizes a construction of \cite{Fock2012}.  When $G$ is not simply-laced, this provides a novel cluster algebraic realization of the $Q$-systems of twisted type, though the cluster structures associated in \cite{DiFrancesco2013} to $Q$-systems of nonsimply-laced untwisted type do not fit into our framework.  We note that in the context of double Bruhat cells, what arises more naturally are the $Y$-system analogues of $Q$-systems, which differ by a standard change of variables.  In different language, we work directly with $\CX$-coordinates rather than cluster variables; this is essential in using amalgamation to form the quotient cluster structures we need.

Given the above result, the sequence of mutations underlying the $Q$-system gives rise to a corresponding sequence of cluster transformations on $G^{c,c}/H$.

\begin{thm*}\emph{(\ref{prop:factorization}, \ref{prop:qsystemintegrability})}
Under the identification of their associated cluster structures, the dynamics of the $Q$-system correspond to those of a certain factorization mapping on the quotient $G^{c,c}/H$.  In particular, these $Q$-systems are discrete integrable in the Liouville sense.
\end{thm*}

Factorization mappings play an important role in discrete integrable systems, analogous to that of Lax forms in continuous-time integrable systems \cite{Deift1989,Moser1991,Veselov1991}.  Given a rule for factoring a group element $g$ as a product $g = hk$, one defines a corresponding factorization mapping by $g \mapsto kh$, typically restricted to some subvariety of $G$.  The factorization relevant for our purposes is defined via the decomposition of an element into opposite Borel subgroups, which is unambiguously defined up to conjugation by $H$.  In addition to making contact with $Q$-systems, the requirement that $c$ be a Coxeter element guarantees that the invariant functions on $G$ descend to form an integrable system on $G^{c,c}/H$, which has a natural symplectic structure \cite{Hoffmann2000}.  The factorization mapping manifestly preserves these invariant functions, hence as observed in \cite{Hoffmann2000} is discrete integrable in the Liouville sense.  The discrete integrability of the corresponding $Q$-system then follows as a corollary of our setup; in type $A_n$ this integrability is well-known from a number of different perspectives \cite{Gekhtman2011,DiFrancesco}.  In fact, $G^{c,c}/H$ is also equipped with an integrable system (a generalization of the relativistic periodic Toda lattice) when $G$ is an affine Kac-Moody group \cite{Williams2012}, leading to a corresponding discrete integrable system based on the analogous factorization mapping.

\begin{thm*}\emph{(\ref{prop:factorization})}
If $G$ is an affine Kac-Moody group, the factorization mapping on $G^{c,c}/H$ is also discrete integrable with conserved quantities derived from characters of finite-dimensional representations.
\end{thm*}

In type $A_{n}^{(1)}$ a generalization of this is treated in \cite{Fock2012}, and is related to the Hirota bilinear difference equation (or octahedron recurrence).  In other simply-laced affine types it is related to the analogues of $Q$-systems for quantum toroidal algebras \cite{Hernandez}.  

Since amalgamation commutes with mutation in a suitable sense, our setup also gives rise to a distinguished sequence of cluster transformations on $G^{c,c}$ itself.  This turns out to be closely related to the Fomin-Zelevinsky twist automorphism, which relates the cluster variables and factorization parametrization associated with a double reduced word.  

\begin{thm*}\emph{(\ref{prop:twistasclustertrans})}
The twist automorphism of $G^{c,c}$ maps the toric chart associated with any seed to the chart obtained from the mutation sequence associated with the factorization mapping on $G^{(c,c)}/H$.  This holds when $G$ is any symmetrizable Kac-Moody group, and yields explicit formulas expressing twisted cluster variables as Laurent monomials in the untwisted cluster variables of a different cluster.  
\end{thm*}

Versions of the twist map exist on many varieties of Lie-theoretic origin with natural cluster structures.  This result parallels similar ones for unipotent cells \cite{Geiss2012} and Grassmannians \cite{Marsh2013}, which show that certain twisted cluster variables differ by a change of coefficients from the untwisted cluster variables obtained from a distinguished sequence of mutations.

Finally, we note that part of the interest in understanding properties of the exchange matrices $B_C$ comes from their appearance (in the simply-laced case) as BPS quivers of pure $\CN=2$ gauge theories \cite{Alim2011}.  The mutation sequences relevant for spectrum computations are in fact iterations of the mutation sequences we consider here.

\textsc{Acknowledgments}  I would like to generously thank Philippe Di Francesco, Rinat Kedem, Nicolai Reshetikhin, Tomoki Nakanishi, Vladimir Fock, and Michael Gekhtman for useful related discussions and remarks while this research was in progress.  This work was supported by NSF grants DMS-12011391 and 0943745, and the Centre for Quantum Geometry of Moduli Spaces at Aarhus University.

\section{Cluster Algebras and Double Bruhat Cells}



\subsection{Cluster Variables and $\CX$-coordinates} 

In this section we fix some basic definitions and facts concerning cluster algebras and $\CX$-coordinates.  More extensive references include \cite{Fomin2006,Fock2003,Gross2013}.  The only nonstandard item is our discussion of amalgamation: while this is usually understood as a gluing operation between seeds \cite{Fock2006}, we will require amalgamations of individual indecomposable seeds.

\begin{defn}
\textbf{(Seeds)} A seed $\Si$ consists of: 
\begin{enumerate}
\item An index set $I = I_f \sqcup I_u$ with a decomposition into frozen and unfrozen indices. 
\item An $I \times I$ exchange matrix $B$ with $B_{ij} \in \BZ$ unless $i, j \in I_f$.  
\item Skew-symmetrizers $d_i \in \BZ_{>0}$ such that $B_{ij}d_j = B_{ji}d_i$.  \qedhere
\end{enumerate}
\end{defn}

\begin{defn}
\textbf{(Mutation)} For any unfrozen index $k$ the mutation of $\Si$ at $k$ is the seed $\mu_k(\Si)$ defined as follows.  It has the same index set, frozen and unfrozen subsets, and skew-symmetrizers as $\Si$. Its exchange matrix $\mu_k(B)$ is given by
\begin{align}\label{eq:matmut}
\mu_k(B)_{ij} = \begin{cases}
-B_{ij} & i = k \text{ or } j=k \\
B_{ij} + \frac12(|B_{ik}|B_{kj} + B_{ik}|B_{kj}|) & i,j \neq k. \xqedhere{3.98cm}{\qedhere}
\end{cases}
\end{align}
\end{defn}

Two seeds $\Si$ and $\Si'$ are said to be mutation equivalent if they are related by a finite sequence of mutations.  Note that the term seed is often taken to include the additional data of an identification of the corresponding cluster variables with a transcendence basis of a fixed function field.  

\begin{defn}
\textbf{(Cluster Variables and $\CX$-coordinates)} To a seed $\Si$ we associate two Laurent polynomial rings $\BC[A_i^{\pm1}]$ and $\BC[X_i^{\pm1}]$, whose generators are indexed by $I$ and referred to as cluster variables and $\CX$-coordinates, respectively.  These are the coordinate rings of two algebraic tori, denoted by $\CA_\Si$ and $\CX_\Si$.  There is a canonical map $p_\Si: \CA_\Si \to \CX_\Si$ defined by $p_\Si^* X_i = \prod_{j \in I} A_j^{B_{ij}}$.  The torus $\CX_\Si$ has a canonical Poisson structure given by
\[
\{X_i,X_j\} = B_{ij}d_jX_iX_j.\qedhere
\]
\end{defn}

While working over the complex numbers is sufficient for our purposes, it is not essential.  Also, what we refer to as $\CX$-coordinates are often called $Y$-variables elsewhere in the literature.  

\begin{remark}\label{remark:lattices}
The ring $\BC[X_i^{\pm1}]$ should be identified with the group ring of the free abelian group $\BZ I$ generated by $I$, and (when $B$ is skew-symmetric) $\BC[A_i^{\pm1}]$ should be identified with the group ring of its dual lattice $(\BZ I)^*$.  In particular, the exchange matrix endows $\BZ I$ with a skew-symmetric form, which is the origin of the map $p_\Si$ and the Poisson structure on $\CX_\Si$.
\end{remark}

\begin{defn}
\textbf{(Cluster Transformations)} To each mutation $\mu_k$ of seeds is associated a pair of rational maps between the corresponding tori, called cluster transformations and also denoted by $\mu_k$.  These satisfy
\[
\begin{tikzcd}
\CA_{\Si} \arrow[dashed]{r}{\mu_k} \arrow{d}{p_{\Si}} & \CA_{\Si'} \arrow{d}{p_{\Si'}} \\
\CX_{\Si} \arrow[dashed]{r}{\mu_k} & \CX_{\Si'}
\end{tikzcd},
\]
where $\Si'=\mu_k(\Si)$, and are defined explicitly by\footnote{Note that our exchange matrix conventions are transpose to those of, for example, \cite{Fomin2006}.}
\begin{gather}\label{eq:Atrans}
\mu_k^*(A'_i) = \begin{cases}
A_i & i \neq k \\
\displaystyle A_k^{-1}\biggl(\prod_{B_{kj}>0}A_j^{B_{kj}} + \prod_{B_{kj}<0}A_j^{-B_{kj}}\biggr) & i = k
\end{cases}
\end{gather}
and
\begin{gather}\label{eq:Xtrans}
\mu_k^*(X'_i) = \begin{cases}
X_iX_k^{[B_{ik}]_+}(1+X_k)^{-B_{ik}} & i \neq k  \\
X_k^{-1} & i = k,
\end{cases}
\end{gather}
where $[B_{ik}]_+\coloneqq\mathrm{max}(0,B_{ik})$.
\end{defn}

The new cluster variables $A'_i$ could also be defined directly as elements of the function field $\BC(\CA_{\Si})$, omitting specific mention of the torus $\CA'_{\Si}$.

\begin{defn}
\textbf{(Cluster Algebras and $\CX$-varieties)} The $\CA$- and $\CX$-spaces $\CA_{|\Si|}$ and $\CX_{|\Si|}$ are the schemes obtained from gluing together along cluster transformations all such tori of seeds mutation equivalent to an initial seed $\Si$.  The map $p_\Si$ extends to a map $p_{|\Si|} \colon \CA_{|\Si|} \to \CX_{|\Si|}$, and the Poisson structure on $\CX_{\Si}$ extends to one on $\CX_{|\Si|}$.  The upper cluster algebra is the algebra of regular functions on $\CA_{|\Si|}$, or equivalently the subalgebra of $\BC(\CA_\Si)$ consisting of elements that are Laurent polynomials in the cluster variables of all seeds mutation equivalent to $\Si$.  It contains the cluster algebra, which is the subalgebra generated by the cluster variables of all such seeds.
\end{defn}

\begin{defn}
\textbf{($\si$-periods)} Let $\wh{\mu} = \mu_{i_1} \circ \cdots \circ \mu_{i_k}$ be a sequence of mutations of a seed $\Si$ and $\si$ a permutation of $I$ such that
\[
\wh{\mu}(B)_{ij} = B_{\si(i)\si(j)}.
\]
In other words, $\wh{\mu}(\Si)$ and $\Si$ are isomorphic after relabeling by $\si$.  Then we say $\wh{\mu}$ is a $\si$-period of $\Si$, or that $\wh{\mu}$ is a mutation-periodic sequence when $\si$ and $\Si$ are understood.  To such a mutation-periodic sequence is associated a pair of rational automorphisms of the tori $\CA_\Si$ and $\CX_\Si$, denoted by $\wh{\mu}_\si$, which we refer to as cluster automorphisms and which are intertwined by the map $p_\Si$.  More formally, these are defined by
\[
\wh{\mu}_\si^*(A_i) = (\mu_{i_1} \circ \cdots \circ \mu_{i_k})^*(A_{\si^{-1}(i)}), \quad \wh{\mu}_\si^*(X_i) = (\mu_{i_1} \circ \cdots \circ \mu_{i_k})^*(X_{\si^{-1}(i)}). \qedhere
\]
\end{defn}

\begin{defn}\label{def:amalg}
\textbf{(Amalgamation)} If $\Si$, $\wt{\Si}$ are seeds and $\pi \colon I \onto \wt{I}$ a surjection of their index sets, we say $\wt{\Si}$ is the amalgamation of $\Si$ along $\pi$ if
\begin{enumerate}
\item For all distinct $i, j \in I$, $\pi(i) = \pi(j)$ implies $i, j \in I_f$ and $B_{ij} = 0$.
\item For all $k, \ell \in \wt{I}$, 
\[
\wt{B}_{k \ell} = \sum_{\mathclap{\substack{i,j: \pi(i) = k,\\ \pi(j) = \ell}}} B_{ij}.
\]
\item $\pi(I_u) \subset \wt{I}_u$.
\item $d_i = d_{\pi(i)}$ for all $i \in I$.  
\end{enumerate}
To such an amalgamation of seeds is associated an amalgamation map $\pi \colon \CX_{\Si} \onto \CX_{\wt{\Si}}$, which is Poisson and defined by
\[
\pi^*(\wt{X}_j) = \prod_{\mathclap{i: \pi(i) = j}} X_i.\qedhere
\]
\end{defn}

In particular, an amalgamation $\wt{\Si}$ of $\Si$ can be associated with any bijection $\vphi: I_1 \isom I_2$ between disjoint subsets of $I_f$ such that $B_{i, \vphi(i)} = 0$ and $d_i = d_{\vphi(i)}$ for all $i \in I_1$.  We set $\wt{I} = I \sminus I_1$, $\wt{I}_u = I_u$, $\wt{I}_f = I_f \sminus I_1$, defining the map $\pi:I \onto \wt{I}$ as the identity on $I \sminus I_1$ and $\vphi$ on $I_1$.  The exchange matrix $\wt{B}$ is then uniquely determined by the hypotheses of \cref{def:amalg}.

\begin{remark}
In the spirit of \cref{remark:lattices}, amalgamation should be understood as deriving from an inclusion of lattices $\BZ\wt{I} \subset \BZ I$, where for each $i \in \wt{I}$ we identify the generator $e_i$ of $\BZ \wt{I}$ with the element $\sum_{\pi(j)=i}e_j$ of $\BZ I$.
\end{remark}

\cref{def:amalg} is somewhat flexible about the relation between frozen and unfrozen subsets of $I$ and $\wt{I}$, and in typical situations we may have $\pi(i)$ be unfrozen though $i$ is frozen.  It is also typically the case that $\Si$ is a direct sum of two other seeds $\Si_1$ and $\Si_2$ (for the obvious notion of direct sum), and the map $\vphi$ identifies some frozen indices of $\Si_1$ with frozen indices of $\Si_2$.  However, our examples require the more general notion given here.  A crucial feature of amalgamations is that under certain mild conditions they commute with cluster transformations.

\begin{prop}\label{prop:amalgcommutes}
Suppose $\wt{\Si}$ is the amalgamation of $\Si$ along $\pi: I \onto \wt{I}$, and that $\pi$ also satisfies the hypotheses of \cref{def:amalg} with respect to $\mu_k(\Si)$ and $\mu_k(\wt{\Si})$ for some unfrozen index $k$. Then $\mu_k(\wt{\Si})$ is also the amalgamation of $\mu_k(\Si)$ along $\pi$, and the respective amalgamation maps and cluster transformations commute:
\[
\begin{tikzcd}
\CX_{\Si} \arrow[dashed]{r}{\mu_k} \arrow[two heads]{d}{\pi} & \CX_{\Si'} \arrow[two heads]{d}{\pi} \\
\CX_{\wt{\Si}} \arrow[dashed]{r}{\mu_k} & \CX_{\wt{\Si}'}.
\end{tikzcd}
\]
\end{prop}
\begin{proof}
For each $i \in \wt{I}$, we must check that $(\pi \circ \mu_k)^* X'_i = (\mu_k \circ \pi)^* X'_i$.  This is clear for $i=k$, while for $i \neq k$ we have
\begin{gather*}
(\pi \circ \mu_k)^* X'_i = \prod_{\mathclap{\pi(j) = i}}(X_jX_k^{[B_{jk}]_+}(1+X_k)^{-B_{jk}})\\
(\mu_k \circ \pi)^* X'_i = (\prod_{\mathclap{\pi(j) = i}}X_j)X_k^{[B_{ik}]_+}(1+X_k)^{-B_{ik}}.
\end{gather*}
Since $B_{ik} = \sum_{\pi(j) = i}B_{jk}$ by assumption, the result follows if
\[
\sum_{\mathclap{\pi(j) = i}}[B_{jk}]_+ = [\sum_{\mathclap{\pi(j) = i}}B_{jk}]_+.
\]
This in turn holds if $B_{jk}$ and $B_{\ell k}$ are of the same sign whenever $\pi(j) = \pi(k) = i$.  But if $B_{jk}$ and $B_{\ell k}$ were of opposite signs for some such $j$, $\ell$, $B'_{j \ell}$ would be nonzero, contradicting our hypothesis about $\pi$.
\end{proof}

When frozen variables of two distinct seeds are glued together by an amalgamation, the assumption that $\pi$ satisfies the needed hypotheses with respect to the mutated seeds always holds.  However, when $\Si$ is not a direct sum this need not be the case.  For example, if $B$ is the adjacency matrix of the quiver
\[
\begin{tikzpicture}
\coordinate [label=1] (1) at (0,0);
\coordinate [label=2] (2) at (2,0);
\coordinate [label=3] (3) at (4,0);
\fill (1) circle (.06);
\fill (2) circle (.06);
\fill (3) circle (.06);
\draw [-to,shorten <=1.3mm,shorten >=1.3mm] (1) to (2);
\draw [-to,shorten <=1.3mm,shorten >=1.3mm] (2) to (3);
\end{tikzpicture}
\]
then we can form an amalgamation by gluing the vertices 1 and 3 together.  However, after mutation at vertex 2, we will have $B'_{13} \neq 0$, hence this is no longer an admissible amalgamation.  

\subsection{Double Bruhat Cells}


We recall some needed facts about double Bruhat cells, following \cite{Fomin1999,Berenstein2005,Williams}.  Let $G$ be a simply-connected semisimple algebraic group, $C$ its Cartan matrix, $W$ its Weyl group, and $B_\pm$ a pair of opposite Borel subgroups.  For each pair $u,v \in W$, there is a double Bruhat cell $G^{u,v}\coloneqq B_+ \dot{u} B_+ \cap B_- \dot{v} B_-$, where $\dot{u}, \dot{v}$ are any representatives of $u$, $v$ in $G$.  Although the $G^{u,v}$ are not cells in the topological sense, they are smooth, rational affine varieties, and are Poisson subvarieties with respect to the standard Poisson-Lie structure on $G$.  

A double reduced word $\mb{i}=(i_1, \dots, i_m)$ for $(u,v)$ is a shuffle of a reduced word for $u$ written in the indices $\{-1,\dots,-r\}$ and a reduced word for $v$ written in the indices $\{1,\dots,r\}$.  For purposes of the following definition we set $i_k = k$ if $k <1$ or $k >m$.  For $k \leq m$ we define $k^+\coloneqq\mathrm{min}\{ 1 \leq \ell \leq m : \ell > k, |i_\ell| = |i_k| \}$, or $k^+ = m+1$ if the set of such $\ell$ is empty.  We let $\ep_k$ be equal to $1$ if $i_k>0$ and $-1$ if $i_k < 0$.

\begin{defn}\label{def:wordseed}
To a double reduced word $\mb{i}=(i_1, \dots, i_m)$ we associate a seed $\Si_{\mb{i}}$.  Its index set $I$ is $\{-r,\dots,-1\} \cup \{1,\dots,m\}$.  An index $k$ is frozen if either $k<0$ or $k^+ > m$.  The exchange matrix $B_{\mb{i}}$ is given by
\begin{align*}
B_{jk} = & \frac{C_{|i_k|,|i_j|}}{2}\biggl(\ep_j[j = k^+] - \ep_k[j^+=k]  + \ep_j[k<j<k^+][j>0] \\ &- \ep_{j^+}[k<j^+<k^+][j^+\leq m] - \ep_k[j<k<j^+][k>0] + \ep_{k^+}[j<k^+<j^+][k^+\leq m] \biggr).
\end{align*}
We let $d_k = d'_{|i_k|}$, where the $d'_i$ are chosen so that $d'_iC_{ij} = d'_jC_{ji}$.  
\end{defn}

The cluster variables and $\CX$-coordinates associated with these seeds have direct geometric realizations on double Bruhat cells.  We let $G_{\Ad}^{u,v}$ denote the double Bruhat cell of the adjoint group $G_{\Ad}$, and $H_{\Ad}$ its Cartan subgroup. Let $\{\om_i^\vee\} \subset \Hom(\BC^*,H_{\Ad})$ denote the fundamental coweights, defined by $\< \om_i^\vee|\al_j\> = \de_{ij}$.  For $t \in \BC^*$, let $t^{\om_i^\vee}$ denote the corresponding element of $H_{\Ad}$.  We fix a choice $\{e_i, f_i\}_{1 \leq i \leq r} \subset \fg$ of positive and negative Chevalley generators, and define simple root subgroups
\[
E_i(t) \coloneqq \exp(te_i), \quad F_i(t) \coloneqq \exp(tf_i).
\]
It will be convenient to set $E_{-i}(t) \coloneqq F_i(t)$, and to abbreviate $E_{\pm i} \coloneqq E_{\pm i}(1)$.  For $w \in W$ we define representatives $\ol{w} \in G$ by setting $\ol{s_i} \coloneqq E_i(-1)E_{-i}(1)E_i(-1)$ and $\ol{w} \coloneqq \ol{s_{i_1}} \cdots \ol{s_{i_n}}$ for any reduced word $s_{i_1} \cdots s_{i_n}$ for $w$. Likewise we define $\dol{w}$ by extension from $\dol{s_i} \coloneqq E_{-i}(-1)E_i(1)E_{-i}(-1)$.  

\begin{defn}\label{defn:Xdef}
For each double reduced word $\mb{i} = (i_1,\dots,i_m)$, let $x_{\mb{i}}: \CX_{\Si_{\mb{i}}} \into G_{\Ad}^{u,v}$ be the open immersion defined by
\begin{gather*}\label{eq:Xmap}
x_{\mb{i}}: (X_{-r},\dots,X_m) \mapsto X^{\om_{r}^\vee}_{-r}\cdots X^{\om_1^\vee}_{-1} E_{i_1} X^{\om_{|i_1|}^\vee}_1 \cdots E_{i_j} X^{\om_{|i_j|}^\vee}_j\cdots E_{i_m}X^{\om_{|i_m|}^\vee}_m.
\end{gather*}
This extends to a map $\CX_{|\Si_{\mb{i}}|} \to G_{\Ad}^{u,v}$ which is Poisson with respect to the canonical Poisson structure on $\CX_{|\Si_{\mb{i}}|}$ and the standard Poisson-Lie structure on $G_{\Ad}^{u,v}$.
\end{defn}

\begin{prop}\label{prop:mixedmoves}
\emph{(\cite{Fock2006})} Suppose that $\mb{i} = (i_1,\dots,i_m)$, $\mb{i}' = (i'_1,\dots,i'_m)$ differ by swapping two adjacent indices differing only by a sign.  That is, for some $1 \leq k < m$, $i_k = -i_{k+1}$, and
\[
i'_\ell = \begin{cases} -i_\ell & \ell = k, k+1 \\ i_\ell & \mathrm{otherwise}. \end{cases}
\]
Then the corresponding sets of $\CX$-coordinates on $G_{\Ad}^{u,v}$ differ by the cluster transformation at $k$:
\[
\begin{tikzcd}
\CX_{\Si_{\mb{i}}} \arrow[dashed]{rr}{\mu_k} \arrow{dr}[swap]{x_{\mb{i}}} && \CX_{\Si_{\mb{i}'}} \arrow{dl}{x_{\mb{i}'}} \\
& G^{u,v}_{\Ad} &
\end{tikzcd}
\]
\end{prop}

To each fundamental weight $\om_i$ is associated an irreducible $G$-representation $L(\om_i)$ with highest-weight vector $v_i$.  The principal generalized minor $\De^{\om_i} \in \BC[G]$ is defined by $\De^{\om_i}(g) = \< v_i | g v_i \>$.  The right-hand side is a matrix coefficient of $L(\om_i)$, and is characterized by the highest-weight component of $gv_i$ being equal to $\< v_i | g v_i \>v_i$.  For any $w,w' \in W$ we define the generalized minor $\De^{\om_i}_{w,w'}$ by $\De^{\om_i}_{w,w'}(g) = \De^{\om_i}(\ol{w}^{-1}g\ol{w'})$, or equivalently $\De^{\om_i}_{w,w'}(g) = \< \ol{w} v_{i} | g \ol{w'} v_{i} \>$.  Given an index $1 \leq k \leq m$ and a double reduced word $\mb{i}$, we define two Weyl group elements
\[
u_{\leq k} \coloneqq s^{\frac12(1-\ep_1)}_{i_1} \cdots s^{\frac12(1-\ep_{k})}_{i_{k}}, \quad v_{>k} \coloneqq  s^{\frac12(\ep_n+1)}_{i_n} \cdots s^{\frac12(\ep_{(k+1)}+1)}_{i_{k+1}},
\]
additionally setting $u_{\leq k} = e$, $v_{>k} = v^{-1}$ for $k < 0$.  

\begin{defn}\label{defn:Aminors}
For each double reduced word $\mb{i} = (i_1,\dots,i_m)$, let $a_{\mb{i}}: \CA_{\Si_{\mb{i}}} \to G^{u,v}$ be the open immersion defined by identifying the generalized minor $\De^{\om_{|i_k|}}_{u_{\leq k},v_{>k}}$ with the cluster variable $A_k$ for $k \in I$.  This extends to a map $\CA_{|\Si_{\mb{i}}|} \to G^{u,v}$ inducing an isomorphism of $\BC[G^{u,v}]$ with the upper cluster algebra $\BC[\CA_{|\Si_{\mb{i}}|}]$.
\end{defn}

The relationship between the geometric realizations of the cluster variables and $\CX$-coordinates associated with $\Si_{\mb{i}}$ is expressed in terms of a certain twist automorphism $\tau$ of $G^{u,v}$.  First denote by $g \mapsto g^\th$ the involution of $G$ determined by $a^\th=a^{-1}$ for $a \in H$, $E_i(t)^\th = E_{-i}(t)$ for $1 \leq i \leq r$, and $g \mapsto g^\io$ the anti-involution determined by $a^\io = a^{-1}$ for $a \in H$, $E_i(t)^\io = E_i(t)$ for $1 \leq i \leq r$.  For generic $g \in G$, we denote by $[g]_\pm$, $[g]_0$ the unique elements of $N_\pm$, $H$ such that  $g = [g]_-[g]_0[g]_+$.  The twist map $\tau$ is then defined by\footnote{Note that the twist map of \cite{Fomin1999} is an isomorphism $G^{u,v} \isom G^{u^{-1},v^{-1}}$ differing from the map here by the antiinvolution $\io$.  However, this merely reflects the fact that the cluster variables introduced in \cite{Berenstein2005} differ from the corresponding functions in \cite{Fomin1999} by $\io$.}
\begin{equation}\label{eq:twist}
\tau(g) \coloneqq \left([\dol{u}g^\io]_-^{-1}\dol{u}g^\io\ol{v}[g^\io\ol{v}]_+^{-1}\right)^\th.
\end{equation}

For each double reduced word $\mb{i}$ we define a modified exchange matrix $B^{\mathrm{mod}} \coloneqq B + M$, where $M$ is the $I \times I$ matrix with entries
\[
M_{jk} = \frac12 C_{|i_k|,|i_j|}\biggl( [j^+,k^+>m] + [j,k <0] \biggr).
\]
The matrix $B^{\mathrm{mod}}$ has integer entries, and the map $p^{\mathrm{mod}}_{\Si_{\mb{i}}}: \CA_{\Si_{\mb{i}}} \to \CX_{\Si_{\mb{i}}}$ defined by $(p^{\mathrm{mod}}_{\Si_{\mb{i}}})^*X_i = \prod_{j \in I}A_j^{B^{\mathrm{mod}}_{ij}}$ extends to a regular map $\CA_{|\Si_{\mb{i}}|} \to \CX_{|\Si_{\mb{i}}|}$.

\begin{prop}\label{prop:dblbruhatensemble}
\emph{(\cite{Williams})} Let $\mb{i}$ be a double reduced word for $u,v \in W$, and $p_G: G^{u,v} \to G_{\Ad}^{u,v}$ the composition of $\tau^{-1}$ and the quotient map $\pi_G$ from $G^{u,v}$ to $G_{\Ad}^{u,v}$.  Then $p_G$ relates the cluster variables and $\CX$-coordinates associated with $\mb{i}$ via a monomial transformation:
\[
\begin{tikzcd}
\CA_{|\Si_{\mb{i}}|} \arrow{r}{a_{|\Si_{\mb{i}}|}} \arrow{d}{p^{\mathrm{mod}}_{\Si_{\mb{i}}}} & G^{u,v} \arrow{d}{p_G} \\
\CX_{|\Si_{\mb{i}}|} \arrow{r}{x_{|\Si_{\mb{i}}|}} & G_{\Ad}^{u,v}.
\end{tikzcd}
\]
In particular, we have the following formula for $\CX$-coordinates in terms of twisted cluster variables:
\[
\pi_G^*X_i = \prod_{j \in I}(\tau^*\!A_j)^{B^{\mathrm{mod}}_{ij}}.
\]
\end{prop}

With some additional notation, this discussion can be extended to include symmetrizable Kac-Moody groups \cite{Williams}.  For a symmetrizable Cartan matrix $C$, let $\wt{r} \coloneqq r + \corank C$ be the dimension of the Cartan subgroup of the corresponding ind-group $G$ \cite{Kumar2002}.  Fix an extension of the coroots $\{\al_i\}_{i=1}^r$ to a basis $\{\al_i\}_{i=1}^{\wt{r}}$ of $\Hom(\BC^*,H)$, inducing a fundamental weight basis $\{\om_i\}_{i=1}^{\wt{r}}$ of $\Hom(H,\BC^*)$.  For $r < i \leq \wt{r}$, define $\al_i \coloneqq d'\sum_{j=1}^r(d'_j)^{-1}\<\al_i^\vee|\al_j\>\om_j$ (where $d'$ is the least common integer multiple of the $d'_j$), inducing a fundamental coweight basis $\{\om_i^\vee\}_{i=1}^{\wt{r}}$ of $\Hom(\BC^*,H) \ot \BQ$.  These choices fix an extension of the Cartan matrix to a nondegenerate, symmetrizable $\wt{r}$-by-$\wt{r}$ matrix, as $C_{ij} = \<\al_i^\vee|\al_j\>$ is now defined for $r < i,j \leq \wt{r}$.  

As $\oplus_{1\leq i \leq \wt{r}}\BZ \al_i$ is a full rank sublattice of $\Hom(H,\BC^*)$, its kernel is a discrete subgroup of $H \cap Z(G)$.  The maximal adjoint form $G_{\Ad}^{\max}$ is the quotient of $G$ by this subgroup (so $\{\om_i\}_{i=1}^{\wt{r}}$ is a basis of its Cartan subgroup's cocharacter lattice), and the minimal adjoint form $G_{\Ad}^{\min}$ is the quotient of $G$ by $Z(G)$ (so $\{\om_i\}_{i=1}^{r}$ is a basis of its Cartan subgroup's cocharacter lattice).  

If $\mb{i}$ is a double reduced word for $u,v$, we have minimal and maximal seeds $\Si_{\mb{i}}^{\min}$, $\Si_{\mb{i}}^{\max}$ with respective index sets
\[
I_{\min} \coloneqq \{-r,\dots,-1\} \cup \{1,\dots,m\}, \quad I_{\max} \coloneqq \{-\wt{r},\dots,-(r+1)\} \cup I_{\min},
\]
and exchange matrices as in \cref{def:wordseed}.  \cref{defn:Xdef} now yields charts $\CX_{\Si^{\min}_{\mb{i}}} \into (G^{\min}_{\Ad})^{u,v}$ and $\CX_{\Si^{\max}_{\mb{i}}} \into (G^{\max}_{\Ad})^{u,v}$, while \cref{defn:Aminors} yields charts $\CA_{\Si^{\min}_{\mb{i}}} \into (G')^{u,v}$ and $\CA_{\Si^{\max}_{\mb{i}}} \into G^{u,v}$ (where $G'$ is the derived subgroup of $G$).  \cref{prop:dblbruhatensemble} now holds in the sense of the following diagram:
\[
\begin{tikzcd}
\CA_{\Si^{\min}_{\mb{i}}} \arrow[hook]{r} \arrow[hook]{d} & \CA_{\Si^{\max}_{\mb{i}}} \arrow[two heads]{r}{p_{\Si_{\mb{i}}^{\max}}^{\mathrm{mod}}} \arrow[hook]{d} & \CX_{\Si^{\max}_{\mb{i}}} \arrow[two heads]{r} \arrow[hook]{d} & \CX_{\Si^{\min}_{\mb{i}}} \arrow[hook]{d} \\
(G')^{u,v} \arrow[hook]{r} & G^{u,v} \arrow[two heads]{r}{p_G} & (G^{\max}_{\Ad})^{u,v} \arrow[two heads]{r} & (G^{\min}_{\Ad})^{u,v}.
\end{tikzcd}
\]
Here the top left and top right maps are induced by the inclusion of lattices $\BZ I_{\min} \into \BZ I_{\max}$ following \cref{remark:lattices}.

\begin{example}
If $C$ is of untwisted affine type, $G'$ is a central extension of the group $L\mr{G}$ of regular maps from $\BC^*$ to a simple Lie group $\mathring{G}$, and $G$ is the semidirect product $G' \rtimes \BC^*$.  $G_{\Ad}^{\max}$ is the quotient of $G$ by $Z(\mr{G})$, embedded as constant maps, while $G_{\Ad}^{\min}$ is the semidirect product$(L\mr{G}/Z(\mr{G})) \rtimes \BC^*$ .
\end{example}

\section{Factorization Dynamics}

In this section we discuss factorization mappings from the perspective of cluster transformations.  To any Cartan matrix $C$ we associate a seed $\Si_C$ with a canonical mutation-periodic sequence.  We realize this seed as an amalgamation of a Coxeter double Bruhat cell, which corresponds to taking its quotient under conjugation by the Cartan subgroup.  We show that the mutation-periodic sequence corresponds to a factorization mapping on this quotient.  In finite type this mapping is known to be discrete integrable \cite{Hoffmann2000}, and we show it is also integrable in affine type.  

\begin{defn}\label{defn:Cseed}
For any symmetrizable $r$-by-$r$ Cartan matrix $C$, let $\Si_C$ be the seed with $I_C = (I_C)_u = \{1,\dots,2r\}$, exchange matrix
\[
B_C \coloneqq \begin{pmatrix} 0 & C^t \\ -C^t & 0 \end{pmatrix},
\]
and $d_i$ derived from the symmetrizers of $C$ in the obvious way.  We let $\wh{\mu}$ be the mutation sequence $\mu_1 \circ \cdots \circ \mu_r$ of $\Si_C$, and $\si$ the permutation of $I$ interchanging $i$ and $i+r$. 
\end{defn}

\begin{prop}
The mutation sequence $\wh{\mu}$ is a $\si$-period of $\Si_C$, that is
\[
\wh{\mu}(B_C)_{ij} = (B_C)_{\si(i)\si(j)}.\qedhere
\]
\end{prop}

\begin{proof}
Since $(B_C)_{\si(i)\si(j)} = -(B_C)_{ij}$, we must check that $\wh{\mu}(B_C) = -B_C$.  This is immediate for the top-left and off-diagonal $r$-by-$r$ blocks of $\wh{\mu}(B_C)$.  We then calculate that
\begin{align*}
\wh{\mu}(B_C)_{i+r,j+r} = & \frac12 \sum_{\mathclap{1\leq k \leq r}} ( C_{k,i}|C_{j,k}| - |C_{k,i}|C_{j,k} ) \\
= & \frac12 \sum_{\mathclap{k = i,j}} ( C_{k,i}|C_{j,k}| - |C_{k,i}|C_{j,k}  ) \\
= & 0.\qedhere
\end{align*}
\end{proof}

Fix a Coxeter element $c = s_1\cdots s_r$ in the Weyl group associated with $C$, and a double reduced word $\mb{i} = (-1,\dots,-r,1,\dots,r)$ for $u = v =c$.  In fact the essential content of this section and the next hold when $u$ and $v$ are possibly distinct Coxeter elements, see \cref{remark:quot}.  When $C$ is not of finite type, $G_{\Ad}$ will refer to the minimal form of the adjoint group associated with $C$, and $\Si_{\mb{i}}$ to the corresponding minimal seed.  Note that $I_{\mb{i}} = \{-1,\dots,-r\} \cup I_C$.

\begin{lemma}\label{lem:bmatrix}
Let $C^t_U$, $C^t_L$ be the upper- and lower-triangular $r \times r$ matrices with 1's on the diagonal such that $C^t_U + C^t_L = C^t$.  That is,
\[
(C^t_U)_{ij} = \de_{ij} + [i<j]C_{ji}, \quad (C^t_L)_{ij} = \de_{ij} + [i>j]C_{ji}.
\]
Then the exchange matrix of $\Si_{\mb{i}}$ has the form
\[
\renewcommand*{\arraystretch}{1.5}
B_{\Si_{\mb{i}}} = \begin{pmatrix}C^t_U - \frac12 C^t & C^t_L & 0 \\ -C^t_U & 0 & -C^t_L \\ 0 & C^t_U & C^t_L - \frac12 C^t\end{pmatrix},
\]
where we have ordered the indices as $-1,\dots,-r,1,\dots,2r$.
\end{lemma}

\begin{proof}
Can be checked directly from \cref{def:wordseed}.
\end{proof}

For any $u, v \in W$, we denote by $G_{\Ad}^{u,v}/H_{\Ad}$ the quotient of $G_{\Ad}^{u,v}$ under conjugation by $H_{\Ad}$, with the following caveat.  If $\mb{j}$ is any double reduced word for $u$, $v$, then since $H_{\Ad}$ is generated by coweight subgroups and $X_{k}^{\om_{|i_k|}^\vee}$ commutes with $E_j$ for $|j| \neq |i_k|$, it follows from the definition of $x_{\mb{j}}$ that the conjugation action of $H_{\Ad}$ preserves the image of $\CX_{\Si_{\mb{j}}}$, and that a good geometric quotient $\CX_{\Si_{\mb{j}}}/H_{\Ad}$ exists.  In fact, from \cref{eq:Xtrans} it is clear that for any seed $\Si'$ mutation-equivalent to $\Si_{\mb{j}}$, the corresponding chart $\CX_{\Si'} \subset G_{\Ad}^{u,v}$ has a good quotient by $H_{\Ad}$.  These charts cover an open subset of $G_{\Ad}^{u,v}$ whose complement is of codimension at least 2, hence this open subset also has a good quotient by $H_{\Ad}$.  The question of whether or not the whole cell $G_{\Ad}^{u,v}$ admits a good quotient will not be relevant for our purposes, so we will simply write $G_{\Ad}^{u,v}/H_{\Ad}$ with the understanding that we may need to restrict to an open subset.

\begin{thm}\label{prop:coxconjquot}
The seed $\Si_C$ is the amalgamation of $\Si_{\mb{i}}$ along the map $\pi: I_{\mb{i}} \onto I_C$ given by 
\[
\pi(k) = \begin{cases} k & k>0 \\ |k| + r & k < 0. \end{cases}
\]
The map $x_{\mb{i}}: \CX_{\Si_{\mb{i}}} \into G_{\Ad}^{c,c}$ descends to an open immersion $\CX_{\Si_C} \into G_{\Ad}^{c,c}/H_{\Ad}$ intertwining the quotient and amalgamation maps:
\[
\begin{tikzcd}
\CX_{\Si_{\mb{i}}} \arrow{r}{x_{\mb{i}}} \arrow[two heads]{d}{\pi} & G_{\Ad}^{c,c} \arrow[two heads]{d}{\pi} \\
\CX_{\Si_C} \arrow{r}{x_{\mb{i}}} & G_{\Ad}^{c,c}/H_{\Ad}.\qedhere
\end{tikzcd}
\]
\end{thm}

\begin{proof}
Using \cref{lem:bmatrix}, one can immediately verify that the hypothesis of \cref{def:amalg} are satisfied by $\Si_C$, $\Si_{\mb{i}}$, and $\pi$.  The conjugation-invariant subalgebra $\BC[\CX_{\Si_{\mb{i}}}]^{H_{\Ad}}$ is manifestly generated by the $X_i$, $X_{-i}X_{i+r}$, and their inverses for $1 \leq i \leq r$.  But this is equal to $\pi^*\BC[\CX_{\Si_C}]$, hence we obtain the map $\CX_{\Si_C} \into G_{\Ad}^{c,c}/H_{\Ad}$. 
\end{proof}

\begin{figure}
\begin{tikzpicture}[decoration={snake,amplitude=1.2}]
\node [matrix] (L) at (0,0)
{
\coordinate [label=below:-2] (-2) at (0,0);
\coordinate [label=below:2] (2) at (2,0);
\coordinate [label=below:4] (4) at (4,0);
\coordinate [label=-1] (-1) at (0,2);
\coordinate [label=1] (1) at (2,2);
\coordinate [label=3] (3) at (4,2);
\foreach \v in {-1,-2,1,2,3,4} {\fill (\v) circle (.06);};
\foreach \s/\t in {-2/-1,3/4} {\draw [-to,dashed,shorten <=1.7mm,shorten >=1.7mm] (\s) to (\t);};
\foreach \s/\t in {-1/1,3/1,-2/2,4/2,1/-2,2/3} {\draw [-to,shorten <=1.7mm,shorten >=1.7mm] (\s) to (\t);};\\
};
\node [matrix] (R) at (7,0)
{
\coordinate [label=below:2] (2) at (0,0);
\coordinate [label=below:4] (4) at (2,0);
\coordinate [label=1] (1) at (0,2);
\coordinate [label=3] (3) at (2,2);
\foreach \v in {1,2,3,4} {\fill (\v) circle (.06);};
\foreach \s/\t in {3/1,4/2} {
  \draw [-to,shorten <=1.7mm,shorten >=1.7mm] ($(\s)+(0,0.06)$) to ($(\t)+(0,0.06)$);
  \draw [-to,shorten <=1.7mm,shorten >=1.7mm] ($(\s)-(0,0.06)$) to ($(\t)-(0,0.06)$);
};
\draw [-to,shorten <=1.7mm,shorten >=1.7mm] (1) to (4);
\draw [shorten <=1.7mm,shorten >=1.7mm] (2) to (1,1);
\draw [-to,shorten <=1.7mm,shorten >=1.7mm] (1,1) to (3);\\
};
\draw [->,decorate] ($(L.east)+(.3,0)$) to node[above]{amalgamation} ($(R.west)+(-.3,0)$);
\end{tikzpicture}
\caption{The quivers of $\Si_{\mb{i}}$ and $\Si_C$ when $C$ is of type $A_2$.  The dashed arrows correspond to entries of $B_{\mb{i}}$ equal to $\pm \frac12$; since they connect frozen vertices they do not affect the structure of cluster transformations, but record the Poisson brackets among frozen variables.  The amalgamation itself ``glues together'' some of the frozen variables: -1 to 3 and -2 to 4.}
\end{figure}
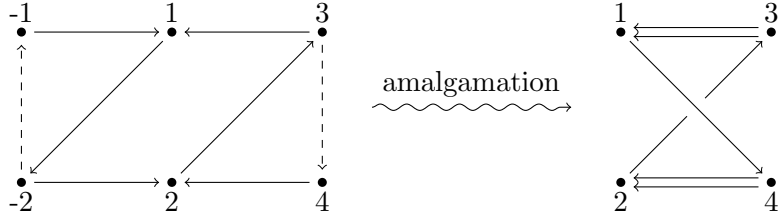

\begin{remark}\label{remark:quot}
If $\mb{j}$ is any double reduced word for $u,v \in W$, the conjugation action of $H_{\Ad}$ on $G_{\Ad}^{u,v}$ will always have a comparably simple expression in the associated $\CX$-coordinates.  However, it is not always the case that quotient map $\CX_{\Si_{\mb{j}}} \onto \CX_{\Si_{\mb{j}}}/H_{\Ad}$ is an amalgamation map.  For example, if $u=c$ but $v=e$, the hypotheses of \cref{def:amalg} will not be satisfied by the quotient map.  However, if $u$ and $v$ are (possibly distinct) Coxeter elements, there will be a unique amalgamation $\wt{\Si}$ of $\Si_{\mb{j}}$ and isomorphism $\CX_{\wt{\Si}} \isom \CX_{\Si_{\mb{j}}}/H_{\Ad}$ intertwining the quotient and amalgamation maps from $\CX_{\Si_{\mb{j}}}$.  In fact, when $u$ and $v$ are Coxeter elements conjugate to $c$, the reader can check that the resulting seed $\wt{\Si}$ is mutation-equivalent to $\Si_C$.  For $GL_n$, this was previously observed (from a different point of view) in \cite{Gekhtman2011}.
\end{remark}

Recall that an integrable system on a (smooth) symplectic variety is a Poisson-commutative subalgebra of its coordinate ring whose differentials generically span Lagrangian subspaces of its cotangent spaces, inducing a Lagrangian foliation of an open subset.  By an integrable system on a Poisson variety we will mean an algebra of functions which restricts to an integrable system on a generic symplectic leaf.

\begin{prop}\label{prop:continuousintegrability}
\emph{(\cite{Hoffmann2000},\cite{Williams})} If $C$ is of finite or affine type, the restrictions of the conjugation-invariant functions on $G_{\Ad}$ form an integrable system on $G^{c,c}_{\Ad}/H_{\Ad}$.  
\end{prop}

\begin{proof}
We only comment that the affine case treated in \cite{Williams} is slightly different from the present one, though the proof there extends straightforwardly.  In loc. cited it was shown that the invariants restrict to form an integrable system on $(G')^{c,c}/H$, where $G'$ is the central extension of the algebraic loop group $L\mr{G}$.  This is actually more delicate, as its symplectic leaves are of dimension $2r+2$, rather than $2r$ (where $r$ is the rank of $\mr{G}$).  For the present case the needed Hamiltonians are derived from the invariant ring $\BC[\mr{G}]^{\mr{G}}$: we pull back this subalgebra along the evaluation map $L\mr{G} \times \BC^* \to \mr{G}$ and take the component invariant under the $\BC^*$ action (in particular they extend to functions on the semidirect product $L\mr{G} \rtimes \BC^*$).  The Hamiltonians for groups of twisted affine type may be produced similarly by embedding them into algebraic loop groups as subgroups invariant under a diagram automorphism.
\end{proof}

\begin{thm}\label{prop:factorization}
The cluster automorphism $\wh{\mu}_{\si}$ of $\CX_{\Si_C}$ coincides with the restriction of the following rational automorphism of $G_{\Ad}^{c,c}/H_{\Ad}$.  Given $g \in G_{\Ad}^{c,c}/H_{\Ad}$, there will generically be unique elements $h_1, h_2 \in H_{\Ad}$ such that, up to conjugation by $H_{\Ad}$,
\[
g = \left((\prod_{\mathclap{1 \leq i \leq r}}^{\curvearrowright} E_i) h_1\right)\left( (\prod_{\mathclap{1 \leq i \leq r}}^{\curvearrowright} F_i) h_2\right).
\]
The rational automorphism of $G_{\Ad}^{c,c}/H_{\Ad}$ is then the factorization mapping
\[
g = \left((\prod_{\mathclap{1 \leq i \leq r}}^{\curvearrowright} E_i) h_1\right)\left( (\prod_{\mathclap{1 \leq i \leq r}}^{\curvearrowright} F_i) h_2\right) \mapsto \left( (\prod_{\mathclap{1 \leq i \leq r}}^{\curvearrowright} F_i) h_2\right)\left((\prod_{\mathclap{1 \leq i \leq r}}^{\curvearrowright} E_i) h_1\right),
\]
taken up to conjugation by $H_{\Ad}$.  Here the product notation indicates we order the terms from left to right by increasing $i$.  In particular, $\wh{\mu}_{\si}$ preserves the restrictions of any conjugation-invariant functions on $G_{\Ad}$, and in finite or affine type is discrete integrable in the Liouville sense.  
\end{thm}

\begin{proof}
By \cref{prop:mixedmoves}, the $\CX$-coordinates on $\CX_{\Si_{\mb{i}}}$ and $\CX_{\Si_{\mb{i}}'}$ (where $\Si_{\mb{i}}' = \wh{\mu}(\Si_{\mb{i}})$) are related by
\begin{multline*}
\left(\prod_{1 \leq i \leq r}^{\curvearrowright} X_{-i}^{\om_i^\vee}\right)\left(\prod_{1 \leq i \leq r}^{\curvearrowright} F_iX_{i}^{\om_i^\vee}\right)\left(\prod_{1 \leq i \leq r}^{\curvearrowright} E_{i}X_{i+r}^{\om_i^\vee}\right) = \\
\left(\prod_{1 \leq i \leq r}^{\curvearrowright} (X'_{-i})^{\om_i^\vee}\right)\left(\prod_{1 \leq i \leq r}^{\curvearrowright} E_{i}(X'_{i})^{\om_i^\vee}\right)\left(\prod_{1 \leq i \leq r}^{\curvearrowright} F_i(X'_{i+r})^{\om_i^\vee}\right).
\end{multline*}
It is straightforward to see that each of the seeds $\mu_k \circ \cdots \circ \mu_r(\Si_{\mb{i}})$ satisfy the hypotheses of \cref{def:amalg} with respect to $\pi: I_{\mb{i}} \onto I_C$, hence we can apply \cref{prop:amalgcommutes} to obtain
\[
\begin{tikzcd}
\CX_{\Si_{\mb{i}}} \arrow[dashed]{r}{\wh{\mu}} \arrow[two heads]{d}{\pi} & \CX_{\Si_{\mb{i}}'} \arrow[two heads]{d}{\pi} \\
\CX_{\Si_C} \arrow[dashed]{r}{\wh{\mu}} & \CX_{\Si_C'}.
\end{tikzcd}
\]
In particular, the $\CX$-coordinates on $\CX_{\Si_C}$ and $\CX_{\Si'_C}$ are related by
\[
\left(\prod_{1 \leq i \leq r}^{\curvearrowright} F_iX_{i}^{\om_i^\vee}\right)\left(\prod_{1 \leq i \leq r}^{\curvearrowright} E_{i}X_{i+r}^{\om_i^\vee}\right) =
\left(\prod_{1 \leq i \leq r}^{\curvearrowright} E_{i}(X'_{i})^{\om_i^\vee}\right)\left(\prod_{1 \leq i \leq r}^{\curvearrowright} F_i(X'_{i+r})^{\om_i^\vee}\right),
\]
up to conjugation by $H_{\Ad}$.

The isomorphism $\CX_{\Si'_C} \isom \CX_{\Si_C}$ given by $\si$ then induces a rational automorphism of $G_{\Ad}^{c,c}/H_{\Ad}$ through
\[
\left(\prod_{1 \leq i \leq r}^{\curvearrowright} E_{i}(X'_{i})^{\om_i^\vee}\right)\left(\prod_{1 \leq i \leq r}^{\curvearrowright} F_i(X'_{i+r})^{\om_i^\vee}\right) \mapsto \left(\prod_{1 \leq i \leq r}^{\curvearrowright} F_i(X'_{i+r})^{\om_i^\vee}\right)\left(\prod_{1 \leq i \leq r}^{\curvearrowright} E_{i}(X'_{i})^{\om_i^\vee}\right).
\]
But this is just the map described in the theorem, with $h_1 = \prod(X'_{i})^{\om_i^\vee}$ and $h_2 = \prod(X'_{i+r})^{\om_i^\vee}$.  That $\wh{\mu}_\si$ preserves invariant functions is clear, hence we obtain discrete integrability in finite and affine types by \cref{prop:continuousintegrability}.  Note that in affine type even though the symplectic leaves of $\CX_{\Si_C}$ are of positive codimension, $\wh{\mu}_\si$ preserves the distinguished symplectic leaf hence restricts to an integrable symplectomorphism of it.
\end{proof}

\section{Q-Systems}\label{sec:q-systems}

$Q$-systems are nonlinear recurrence relations associated with affine Dynkin diagrams $X_N^{(\ka)}$.  We review their normalized versions and cluster-algebraic realizations following \cite{Kedem2013,DiFrancesco2013}, which we extend to include twisted types.  In twisted and simply-laced untwisted types these systems are encoded by the seeds $\Si_C$ studied in the previous section.  The $Q$-system itself is realized by a sequence of cluster transformations coinciding with that of the corresponding factorization mapping, though realized by cluster variables rather than $\CX$-coordinates.  Since the relevant exchange matrix is nondegenerate, the two sets of variables differ by a finite map, leading to the discrete integrability of these $Q$-systems.

Recall that affine Dynkin diagrams are classified by pairs of a finite-type diagram $X_N$ and an automorphism of order $\ka$.  This induces an automorphism of the simple Lie algebra of type $X_N$, whose invariant subalgebra is also simple and whose type we denote by $Y_M$.  Clearly for untwisted types ($\ka = 1$) we have $X_N = Y_M$, while for twisted types the correspondence is given below.  It is summarized by the fact that the Langlands dual of $X_N^{(\ka)}$ is the affinization of the Langlands dual of $Y_M$.  

\begin{center}
\renewcommand*{\arraystretch}{1.5}
\begin{tabular}{ c || c | c | c | c }
$X_N^{(\ka)}$ & $A_{2r-1}^{(2)}$ & $D_{r+1}^{(2)}$ & $E_6^{(2)}$ & $D_4^{3}$ \\
\hline
$Y_M$ & $C_r$ & $B_r$ & $F_4$ & $G_2$ \\
\end{tabular}
\end{center}

\begin{defn}
The $Q$-system of type $X_N^{(\ka)}$ is the following recurrence relation in the commuting variables $\{Q_n^{(a)}\}$, where $n \in \BZ$ is a discrete ``time'' variable and $a$ is an index labeled by the roots of $Y_M$.  If $X_N^{(\ka)}$ is of untwisted simply-laced type and $C$ the Cartan matrix of type $X_N$, the corresponding $Q$-system is
\begin{gather*}
(Q^{(a)}_{n})^2 = Q^{(a)}_{n-1}Q^{(a)}_{n+1} + \prod_{b \neq a} (Q^{(a)}_{n})^{-C_{ba}}.
\end{gather*}
For $X_N^{(\ka)}$ of twisted type, the corresponding $Q$-systems are as follows \cite{Hatayama2002,Hernandez2010}:
\begin{align*}
A_{2r-1}^{(2)} &\begin{cases}
(Q^{(a)}_{n})^2 = Q^{(a)}_{n-1}Q^{(a)}_{n+1}+Q^{(a-1)}_{n}Q^{(a+1)}_{n} & \hspace{23.5pt} 1 \leq a < r \\
(Q^{(r)}_{n})^2 = Q^{(r)}_{n-1}Q^{(r)}_{n+1} + (Q^{(r)}_{n})^2 & 
\end{cases} \\
D_{r+1}^{(2)} &\begin{cases}
(Q^{(a)}_{n})^2 = Q^{(a)}_{n-1}Q^{(a)}_{n+1}+Q^{(a-1)}_{n}Q^{(a+1)}_{n} & 1 \leq a < r-1 \\
(Q^{(r-1)}_{n})^2 = Q^{(r-1)}_{n-1}Q^{(r-1)}_{n+1} + Q^{(r-2)}_{n}(Q^{(r)}_{n})^2 & \\
(Q^{(r)}_{n})^2 = Q^{(r)}_{n-1}Q^{(r)}_{n+1} + Q^{(r-1)}_{n}& 
\end{cases} \\
E_6^{(2)} & \begin{cases}
(Q^{(1)}_{n})^2 = Q^{(1)}_{n-1}Q^{(1)}_{n+1} + Q^{(2)}_{n} & \\
(Q^{(2)}_{n})^2 = Q^{(2)}_{n-1}Q^{(2)}_{n+1} + Q^{(1)}_{n}Q^{(3)}_{n} & \\
(Q^{(3)}_{n})^2 = Q^{(3)}_{n-1}Q^{(3)}_{n+1} + (Q^{(2)}_{n})^2Q^{(4)}_{n} & \\
(Q^{(4)}_{n})^2 = Q^{(4)}_{n-1}Q^{(4)}_{n+1} + Q^{(3)}_{n} & 
\end{cases} \\
D_4^{3} & \begin{cases}
(Q^{(1)}_{n})^2 = Q^{(1)}_{n-1}Q^{(1)}_{n+1} + Q^{(2)}_{n} & \\
(Q^{(2)}_{n})^2 = Q^{(2)}_{n-1}Q^{(2)}_{n+1} + (Q^{(1)}_{n})^3 & 
\end{cases}
\end{align*}
Here we set $Q^{(0)}_n = 1$ and enumerate the roots of $Y_M$ as in \cref{fig:twistedlabels}.
\end{defn}

\begin{figure}
\begin{tikzpicture}
\newcommand*{\Drad}{1}
\newcommand*{\Ddotsdist}{1}
\newcommand*{\DrawDots}[1]{
  \fill ($(#1) + .25*(\Ddotsdist,0)$) circle (.03);
  \fill ($(#1) + .5*(\Ddotsdist,0)$) circle (.03);
  \fill ($(#1) + .75*(\Ddotsdist,0)$) circle (.03);
}
\newcommand*{\DrawDoubleEdge}[1]{
  \draw ($(#1)+(0,.06)$) -- ($(#1)+(\Drad,.06)$);
  \draw ($(#1)+(0,-.06)$) -- ($(#1)+(\Drad,-.06)$);
}
\newcommand*{\DrawTripleEdge}[1]{
  \draw ($(#1)+(0,.06)$) -- ($(#1)+(\Drad,.06)$);
  \draw ($(#1)+(0,0)$) -- ($(#1)+(\Drad,0)$);
  \draw ($(#1)+(0,-.06)$) -- ($(#1)+(\Drad,-.06)$);
}
\newcommand*{\DrawLessThan}[1]{
  \draw ($(#1)+(.5*\Drad,0)+(.1,.2)$) -- ++(-.2,-.2) -- ++(.2,-.2);
}
\newcommand*{\DrawGreaterThan}[1]{
  \draw ($(#1)+(.5*\Drad,0)+(-.1,.2)$) -- ++(.2,-.2) -- ++(-.2,-.2);
}
\newcommand*{\DrawVertex}[1]{\draw [fill=white] (#1) circle (.12);}
\newcommand*{\BlankVertex}[1]{\path [fill=white] (#1) circle (.12);}
\newcommand*{\LabelBelow}[2]{
  \node at ($(#1)-(0,.4)$) {{\small $#2$}};
}
\newcommand*{\LabelRight}[2]{
  \node at ($(#1)+(.4,0)$) {{\small $#2$}};
}
\node [matrix] (A) at (0,0)
{
\coordinate (0) at (\Drad,\Drad);
\coordinate (1) at (0,0);
\coordinate (2) at (\Drad,0);
\coordinate (3) at (2*\Drad,0);
\coordinate (r-2) at (2*\Drad+\Ddotsdist,0);
\coordinate (r-1) at (3*\Drad+\Ddotsdist,0);
\coordinate (r) at (4*\Drad+\Ddotsdist,0);
\draw (1) -- (3);
\draw (0) -- (2);
\draw (r-2) -- (r-1);
\DrawDoubleEdge{r-1};
\DrawLessThan{r-1};
\DrawDots{3};
\foreach \v in {3,r-2} {\BlankVertex{\v}};
\foreach \v in {0,1,2,r-1,r} {\DrawVertex{\v}};
\foreach \v in {1,2,r-1,r} {\LabelBelow{\v}{\v}};
\LabelRight{0}{0}\\
};
\node [matrix] (D) at (7,0)
{
\coordinate (0) at (0,0);
\coordinate (1) at (\Drad,0);
\coordinate (2) at (2*\Drad,0);
\coordinate (r-2) at (2*\Drad+\Ddotsdist,0);
\coordinate (r-1) at (3*\Drad+\Ddotsdist,0);
\coordinate (r) at (4*\Drad+\Ddotsdist,0);
\draw (1) -- (2);
\draw (r-2) -- (r-1);
\DrawDoubleEdge{0};
\DrawLessThan{0};
\DrawDoubleEdge{r-1};
\DrawGreaterThan{r-1};
\DrawDots{2};
\foreach \v in {2,r-2} {\BlankVertex{\v}};
\foreach \v in {0,1,r-1,r} {\DrawVertex{\v}};
\foreach \v in {0,1,r-1,r} {\LabelBelow{\v}{\v}};\\
};
\node [matrix] (E6) at (0,-2)
{
\foreach \v in {0,...,4} {\coordinate (\v) at (\v*\Drad,0);};
\draw (0) -- (2);
\draw (3) -- (4);
\DrawDoubleEdge{2};
\DrawLessThan{2};
\foreach \v in {0,...,4} {\DrawVertex{\v} \LabelBelow{\v}{\v}};\\
};
\node [matrix] (D4) at (7,-2)
{
\foreach \v in {0,1,2} {\coordinate (\v) at (\v*\Drad,0);};
\draw (0) -- (1);
\DrawTripleEdge{1};
\DrawLessThan{1};
\foreach \v in {0,1,2} {\DrawVertex{\v} \LabelBelow{\v}{\v};}\\
};
\node at ($(A.west)+(-.6,.23)$) {$A_{2r-1}^{(2)}$};
\node at ($(D.east)+(.6,.23)$) {$D_{r+1}^{(2)}$};
\node at ($(E6.west)+(-.5,.23)$) {$E_{6}^{(2)}$};
\node at ($(D4.east)+(.6,.23)$) {$D_{4}^{(3)}$};
\end{tikzpicture}
\caption{Affine Dynkin diagrams of twisted type and enumerations of their vertices.  The diagram $Y_M$ is the subdiagram whose nodes have nonzero labels.}
\label{fig:twistedlabels}
\end{figure}
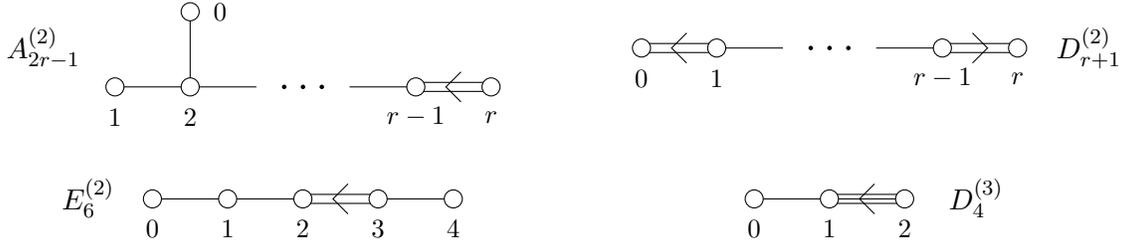



We omit the definition of the $Q$-systems of nonsimply-laced untwisted type, as they lie outside the scope of our main result.  Also absent from the above discussion is the twisted type $A_{2n}^{(2)}$; its relationship with the corresponding finite type is more subtle, and it does not admit an interpretation in terms of cluster transformations.\footnote{It contains the relation $(Q^{(r)}_{n})^2 = Q^{(r)}_{n-1}Q^{(r)}_{n+1} + Q^{(r-1)}_{n}Q^{(r)}_{n}$, whose terms cannot be rearranged into an exchange relation since $Q^{(r)}_{n}$ appears on both sides.}  Thus when referring to a generic twisted type $X_N^{(\ka)}$ we will tacitly assume it is not of type $A_{2n}^{(2)}$.

The correspondence between $X_N^{(\ka)}$ and $Y_M$ allows us to write the above $Q$-systems uniformly as follows:

\begin{prop}
Let $X_N^{(\ka)}$ be of twisted type or simply-laced untwisted type, and $C$ the Cartan matrix of the associated finite type $Y_M$.  Then the $Q$-system of type $X_N^{(\ka)}$ may be written as
\begin{gather*}
(Q^{(a)}_{n})^2 = Q^{(a)}_{n-1}Q^{(a)}_{n+1} + \prod_{b \neq a} (Q^{(a)}_{n})^{-C_{ba}}.
\end{gather*}
\end{prop}

\begin{proof}
Follows by inspection of the above list and the definition of $Y_M$.
\end{proof}

To realize $Q$-systems in terms of cluster transformations, it is convenient to replace them with certain normalized, but equivalent, $Q$-systems.  These normalized variables differ from those of the usual $Q$-system via rescaling by certain roots of unity.  

\begin{prop}
\emph{(\cite{Kedem2013,DiFrancesco2013})} The normalized $Q$-system
\begin{align}\label{eq:normqsys}
\wt{Q}^{(a)}_{n-1}\wt{Q}^{(a)}_{n+1} = (\wt{Q}^{(a)}_{n})^2 + \prod_{b \neq a} (\wt{Q}^{(b)}_{n})^{-C_{ba}}
\end{align}
is equivalent to the ordinary $Q$-system under the rescaling $\wt{Q}^{(a)}_{n} = \ep_{a}Q^{(a)}_{n}$, where the $\ep_{a} \in \BC$ are defined by $\prod_{1 \leq a \leq r} \ep_{a}^{C_{ab}}=-1$ for all $1 \leq b \leq r$.  
\end{prop}

\begin{proof}
Note that the existence of such $\ep_a$ follows from the nondegeneracy of $C$.  The derivation of \cref{eq:normqsys} is then straightforward.
\end{proof}

\begin{remark}
The normalized $Q$-systems also have a direct interpretation in terms of $T$-systems.  These are relations among $q$-characters of Kirillov-Reshetikhin modules, in variables $\{T_n^{(a)}(u)\}$ where $n$ and $a$ are as before and $u \in \BC$ is a spectral parameter.  In the simply-laced case, the relations are
\[
T^{(a)}_{n}(u+1)T^{(a)}_{n}(u-1) = T^{(a)}_{n-1}(u)T^{(a)}_{n+1}(u) + \prod_{b \neq a} (T^{(b)}_{n}(u))^{-C_{ba}}.
\]
By forgetting the spectral parameter $u$, we obtain the usual $Q$-system, but by forgetting instead the parameter $n$ we obtain the normalized $Q$-system.  A similar statement holds for the twisted case, with some subtlety in that we must only consider $u$ modulo a certain additive constant.  
\end{remark}

Given a finite-type Cartan matrix $C$, we let $A^{(1)}_k,\dots,A^{(2r)}_k$ denote the cluster variables associated with the seed $\wh{\mu}^k_\si(\Si_C)$ for $k \in \BZ$.  Recall from \cref{defn:Cseed} that the exchange matrix of $\Si_C$ is 
\[
B_C \coloneqq \begin{pmatrix} 0 & C^t \\ -C^t & 0 \end{pmatrix},
\]
the mutation sequence $\wh{\mu}$ is $\mu_1 \circ \cdots \circ \mu_r$, and $\si$ interchanges $i$ and $i+r$.  As elements of the (upper) cluster algebra $\BC[\CA_{|\Si_C|}]$ the relations among the $A^{(i)}_k$ are in fact equivalent to normalized $Q$-systems under the identification $A^{(i)}_k \mapsto \wt{Q}^{(i)}_k$.  Note that $A^{(i+r)}_k = A^{(i)}_{k+1}$ for $1 \leq i \leq r$, so we lose no information by restricting our attention to $A^{(1)}_k,\dots,A^{(r)}_k$.   

\begin{thm}\label{prop:qsystemsasclusters}
Let $C$ be a finite-type Cartan matrix, and $A^{(1)}_k,\dots,A^{(r)}_k$ cluster variables associated with $\wh{\mu}^k_\si(\Si_C)$.
\begin{enumerate}
\item \emph{(\cite{Kedem2013,DiFrancesco2013})} If $C$ is of simply-laced type $X_N$, the relations among the cluster variables $A_k^{(i)}$ coincide with those of the normalized $Q$-system of type $X_N^{(1)}$.
\item If $C$ is of nonsimply-laced type $Y_M$, the relations among the cluster variables $A_k^{(i)}$ coincide with those of the normalized $Q$-system of the associated twisted type $X_N^{(\ka)}$.
\end{enumerate}
\end{thm}

\begin{proof}
Given the definition of the normalized $Q$-systems in \cref{eq:normqsys}, this is a straightforward check involving the definition of the exchange matrix $B_C$ and the cluster automorphism $\wh{\mu}_\si$.
\end{proof}

\begin{thm}\label{prop:qsystemintegrability}
For $X_N^{(\ka)}$ of twisted type or simply-laced untwisted type, the corresponding $Q$-system is discrete integrable in the Liouville sense.
\end{thm}

\begin{proof}
The statement should be understood in light of \cref{prop:qsystemsasclusters}, which says that incrementing the discrete time variable $n$ of the (normalized) $Q$-system is equivalent to expanding the rational symplectomorphism $\wh{\mu}_\si$ of $\CA_{\Si_C}$ in terms of cluster variables.  Since the matrix $B_C$ is nondegenerate, the canonical map $p_{\Si_C}: \CA_{\Si_C} \to \CX_{\Si_C}$ is a finite cover.  In particular, $\CA_{\Si_C}$ inherits from $\CX_{\Si_C}$ a symplectic structure and the integrable system of \cref{prop:continuousintegrability}.  Since $p_{\Si_C}: \CA_{\Si_C} \to \CX_{\Si_C}$ intertwines the associated automorphisms $\wh{\mu}_\si$ of $\CA_{\Si_C}$ and $\CX_{\Si_C}$, and the latter preserves the integrable system on $\CX_{\Si_C}$ by \cref{prop:factorization}, the former is also discrete integrable.  Since the normalized and unnormalized $Q$-systems differ by an invertible rescaling, the integrability of the normalized $Q$-system implies that of the unnormalized version.
\end{proof}

\section{The Twist Automorphism}

Since amalgamation commutes with mutation, the mutation sequence of $\Si_C$ studied in the previous sections lifts to a mutation sequence on the double Bruhat cell $G^{c,c}$ itself.  We now show that this sequence is intimately connected with the twist automorphism of $G^{c,c}$.  Specifically, any two clusters related by the corresponding sequence of cluster transformations are also mapped to each other by the twist automorphism.  Equivalently, the twist pulls back cluster variables to cluster monomials of the seed obtained by this mutation sequence.  While these pullbacks are generally not cluster variables, the unfrozen cluster variables are taken to monomials with only a single unfrozen factor, so in this sense the twist acts by a change of coefficients.  From the perspective of Poisson geometry this is quite natural; it is known that the twist automorphism is Poisson \cite{Gekhtman2002}, hence both twisted and untwisted cluster variables have quadratic brackets with respect to the standard Poisson-Lie structure.   

\begin{thm}\label{prop:twistasclustertrans}
Let $G$ be a symmetrizable Kac-Moody group, $\tau$ the twist automorphism of $G^{c,c}$, and $\CA_{\Si} \subset G^{c,c}$ the toric chart associated with a seed $\Si$.  Then $\tau$ restricts to an isomorphism of $\CA_{\Si}$ onto $\CA_{\wh{\mu}(\Si)}$, where $\wh{\mu} = \mu_1 \circ \cdots \mu_r$ is the mutation sequence consisting of a single mutation at each unfrozen index.  In particular, if $\{A_i\}$ and $\{A'_i\}$ are the cluster variables associated with $\Si$ and ${\wh{\mu}(\Si)}$, respectively, then the $\{A'_i\}$ and the twisted cluster variables $\{\tau^*(A_i)\}$ are Laurent monomials in each another.  If $\Si$ is the seed associated with the double reduced word $\mb{i} = (-1,\dots,-r,1,\dots,r)$, this transformation is explicitly given by

\[
A'_i = \prod_{j \in I} (\tau^* A_j)^{M_{ij}},
\]
where $M$ is the $I \times I$ matrix with entries
\[
M_{j,k} = \begin{cases*} \< \om_{|i_j|} | \al_{|i_k|}^\vee \> \quad(= \de_{jk}) & $1 \leq j, k \leq r$ \\
\< c\om_{|i_j|} | \al_{|i_k|}^\vee \> & $j > r$ and $k < 0$ \\
\< c^{-1}\om_{|i_j|} | \al_{|i_k|}^\vee \> & $j < 0$, and $k > r$ or $k < -r$ \\
0 & otherwise.\end{cases*}
\]
\end{thm}

\begin{proof}
From \cref{lem:untwistedeval} and \cref{prop:dblbruhatensemble} it follows immediately that
\[
A'_i = \prod_{j \in I} (\tau^* A_j)^{(NB^{\mathrm{mod}}_{\Si})_{ij}},
\]
where $N$ is the matrix of \cref{lem:untwistedeval} and $B^{\mathrm{mod}}_{\Si}$ is the modified exchange matrix associated with $\Si$ as in \cref{prop:dblbruhatensemble}.  Most of the difficulty in verifying that the product of $N$ and $B^{\mathrm{mod}}_{\Si}$ is the given matrix $M$ is encapsulated in \cref{lem:coxetereq}.  For example, for $1 \leq i,k \leq r$, we may use it to compute
\begin{align*}
(NB^{\mathrm{mod}}_{\Si})_{i+r,-k} &= \< (c\om_i) - \om_i | \om_k^{\vee} + \sum_{j<k} C_{kj} \om_j^{\vee} \>\\
&= \< (c\om_i) - \om_i | \al_k^{\vee} - (\om_k^{\vee} + \sum_{j>k} C_{kj} \om_j^{\vee}) \>\\
&= \< (c\om_i) - \om_i | \al_k^{\vee} \> + \de_{ik}\\
&= \< c\om_i | \al_k^{\vee}\>.
\end{align*}

Given that $M = NB^{\mathrm{mod}}_{\Si}$, the theorem follows by verifying that $M$ satisfies the hypotheses of \cref{lem:changeofcoefficients} with respect to the exchange matrices $B_{\Si}$ and $B_{\wh{\mu}(\Si)}$.  Note that $B_{\wh{\mu}(\Si)} = -B_{\Si}$, as $\wh{\mu}(\Si)$ is associated with the double reduced word $(1,\dots,r,-1,\dots,-r)$.  This computation then parallels that of $M$ itself, again with \cref{lem:coxetereq} being the core of the calculation.
\end{proof}

\begin{remark}
If $C$ is of finite type, the decomposition of $M$ into $r$-by-$r$ blocks is
\[
M = \begin{pmatrix} 0 & 0 & c^{-1} \\ 0 & \mathrm{Id} & 0 \\ c & 0 & 0 \end{pmatrix}.
\]
Here we express $c$ as a matrix via its action on the fundamental weight basis, and order the indices by $(-1,\dots,-r,1,\dots,2r)$.
\end{remark}

\begin{lemma}\label{lem:changeofcoefficients}
Let $\Si$, $\wt{\Si}$ be two seeds with the same index set $I$ and unfrozen subset $I_u$.  For an invertible $I \times I$ matrix $M$, let $\vphi_M: \CA_{\wt{\Si}} \isom \CA_{\Si}$ be the isomorphism defined by
\begin{equation}\label{eq:vphiM}
\vphi_M^*( A_i ) = \prod_{j \in I} \wt{A}_j^{M_{ij}}.
\end{equation}
Suppose that $M$ satisfies the following conditions:
\begin{enumerate}
 \item $\wt{B}_{ij} = (BM)_{ij}$ when $i$ is unfrozen. 
\item $M_{ij} = \de_{ij}$ when $j$ is unfrozen.
\end{enumerate}
In particular $B_{ij} = \wt{B}_{ij}$ when $i$ and $j$ are both unfrozen, hence $\Si$ and $\wt{\Si}$ are of the same cluster type.  Then we have:
\begin{enumerate}
\item
The map $\vphi_M$ extends to an isomorphism between $\CA_{\mu_k(\wt{\Si})}$ and $\CA_{\mu_k(\Si)}$ for any unfrozen index $k$.  Specifically, if $M'$ is the $I \times I$ matrix defined by
\begin{equation}\label{eq:M'}
M'_{ij} = \begin{cases} M_{ij} & i \neq k \\ 2\de_{kj} - M_{kj} + \sum_{\ell \in I}([B_{k\ell}M_{\ell j}]_- - [B_{k\ell}]_-M_{\ell j})& i = k,\end{cases}
\end{equation}
then the corresponding isomorphism $\vphi_{M'}: \CA_{\mu_k(\wt{\Si})} \isom \CA_{\mu_k(\Si)}$ satisfies
\[
\begin{tikzcd}
\CA_{\wt{\Si}} \arrow{r}{\vphi_M} \arrow[dashed]{d}{\mu_k} & \CA_{\Si} \arrow[dashed]{d}{\mu_k} \\
\CA_{\wt{\Si}'} \arrow{r}{\vphi_{M'}} & \CA_{\Si'}.
\end{tikzcd}
\]
\item
If $B_{ij}=0$ when $i$ and $j$ are both unfrozen (so $\Si$, $\wt{\Si}$ are of cluster type $A_1^n$), then $\vphi_M$ extends to an isomorphism of $\CA$-spaces and upper cluster algebras.  
\end{enumerate}
\end{lemma}

\begin{proof}
To prove the first claim one must check that for any cluster variable $A'_i$ on $\CA_{\Si'}$, we have $\vphi_{M'}^*A'_i = (\mu_k \circ \vphi_M \circ \mu_k)^*A'_i$.  The condition that $M_{ij} = \de_{ij}$ when $j$ is unfrozen ensures this holds for $i \neq k$.  The condition that $\wt{B}_{kj} = (BM)_{kj}$ ensures $(\mu_k \circ \vphi_M \circ \mu_k)^*A'_k$ is a Laurent monomial in the cluster variables on $\CA_{\wt{\Si}'}$, and the given formula for $M'$ follows from explicitly calculating this composition using \cref{eq:Atrans,eq:vphiM}.

The second claim follows inductively once we establish that $M'$ satisfies the same hypotheses as $M$, but with respect to the seeds $\Si'$, $\wt{\Si}'$.  That $M'_{ij} = \de_{ij}$ when $j$ is unfrozen can be checked generally without any assumptions on the cluster type of $\Si$.  On the other hand, a direct computation reveals that $B_{ij}$ vanishing when $i$ and $j$ are unfrozen is a sufficient condition to ensure $\wt{B}'_{ij} = (B'M')_{ij}$ when $i$ is unfrozen.
\end{proof}

When $C$ is not of finite type, we take $G_{\Ad}$ to be the maximal form of the adjoint group in the following statement.

\begin{lemma}\label{lem:untwistedeval}
Let $\CX_{\mb{i}} \subset G_{\Ad}^{c,c}$ be the toric chart associated with the double reduced word $\mb{i} = (-1,\dots,-r,1,\dots,r)$, and $\CA_{\mb{i}'} \subset G^{c,c}$ the chart associated with $\mb{i}' = (1,\dots,r,-1,\dots,-r)$.  Then the quotient map $\pi: G^{c,c} \onto G_{\Ad}^{c,c}$ restricts to a finite cover of $\CA_{\mb{i}'}$ onto $\CX_{\mb{i}}$.  Equivalently, the (pullbacks to $G^{c,c}$ of the) $\CX$-coordinates associated with $\mb{i}$ are Laurent monomials in the untwisted cluster variables associated with $\mb{i}'$.  In fact,
\[
A_i = \prod_{j \in I} (\pi^*X_j)^{N_{ij}},
\]
where
\[
N_{jk} = \begin{cases*} \< c \om_{|i_j|} | \om_{|i_k|}^\vee \> & $j > r$, $k < 0$ \\
\< c^{-1} \om_{|i_j|} | \om_{|i_k|}^\vee \> & $j < 0$, $k > r$ \\
\< \om_{|i_j|} | \om_{|i_k|}^\vee \> & otherwise. \end{cases*}
\]
\end{lemma}

\begin{proof}
By \cref{defn:Aminors} the cluster variables associated with $\mb{i}'$ are generalized minors of the form $\De_{e,c^{-1}}^{\om_i}$, $\De_{e,e}^{\om_i}$, and $\De_{c,e}^{\om_i}$.  Calculating the matrix $N$ consists of evaluating such minors on an element of the form
\[
g=\left( \prod^{\curvearrowright}_{1 \leq i \leq \wt{r}} X_{-i}^{\om_i^{\vee}} \right) \left(\prod^{\curvearrowright}_{1 \leq i \leq r} F_i X_i^{\om_i^{\vee}}\right)\left(\prod^{\curvearrowright}_{1 \leq i \leq r} E_i X_{i+r}^{\om_i^{\vee}}\right).
\]
This involves fractional powers of the $X_i$, since the coweight subgroups themselves do not act on the fundamental representations, but only covering groups of them.  

By definition $\De_{e,c^{-1}}^{\om_i}(g) = \< v_i | g \ol{s_r} \cdots \ol{s_1} v_i \>$, where $v_i$ is a highest weight vector of the fundamental representation of highest weight $\om_i$.  The key point is that while the action of $E_i$ or $F_i$ on a vector of weight $\om$ is in general a sum of components with weights of the form $\om + n\al_i$, many of these can be discarded in the computation of a given generalized minor.  For example, one can check inductively that for $1 \leq k \leq r$,
\begin{align*}
\De_{e,c^{-1}}^{\om_i}(g) = \< v_i | \left(\prod^{\curvearrowright}_{1 \leq i \leq \wt{r}} X_{-i}^{\om_i^{\vee}}\right) \left(\prod^{\curvearrowright}_{1 \leq i \leq r} F_i X_i^{\om_i^{\vee}}\right) \left(\prod^{\curvearrowright}_{1 \leq i \leq k} E_i X_{i+r}^{\om_i^{\vee}}\right) \ol{s_k} \cdots \ol{s_1} v_i \> \left(\prod_{j = k+1}^{r} X_{j+r}^{\< s_{j} \cdots s_1 \om_i | \om_j^{\vee} \>}\right),
\end{align*}
and from this that
\begin{align*}
\De_{e,c^{-1}}^{\om_i}(g) &= \< v_i | \left(\prod^{\curvearrowright}_{1 \leq i \leq\wt{r}} X_{-i}^{\om_i^{\vee}}\right)\left(\prod^{\curvearrowright}_{1 \leq i \leq r} F_i X_i^{\om_i^{\vee}}\right) v_i \> \left(\prod_{j=1}^{r} X_{j+r}^{\< s_{j} \cdots s_1 \om_i | \om_j^{\vee} \>}\right) \\
&= \left(\prod_{j=1}^{\wt{r}} X_{-j}^{\< \om_i | \om_j^{\vee} \>}\right)\left(\prod_{j=1}^{r} X_{j}^{\< \om_i | \om_j^{\vee} \>}\right)\left(\prod_{j=1}^{r} X_{j+r}^{\< s_{j} \cdots s_1 \om_i | \om_j^{\vee} \>}\right)
\end{align*}
Since 
\[
\< c^{-1} \om_i | \om_j^{\vee} \> = \< s_{j} \cdots s_1 \om_i | \om_j^{\vee} \> = \< s_{r} \cdots s_1 \om_i | \om_j^{\vee} \>,
\]
we obtain the stated values of $N_{jk}$ when $j < 0$.  Note that up to a scalar factor this expression depends on choosing $\ol{s_i}$ as the representative of $s_i$ in $G$.  The remaining entries of $N$ can be computed following the same logic.
\end{proof}

\begin{lemma}\label{lem:coxetereq}
For $1 \leq i,k \leq r$, the Coxeter element $c = s_1 \cdots s_r$ satisfies
\begin{gather*}
\< (c \om_i) - \om_i | \om_k^\vee + \sum_{j>k} C_{kj}\om_j^\vee \> = - \de_{ik}, \\
\< (c^{-1} \om_i) - \om_i | \om_k^\vee + \sum_{j<k} C_{kj}\om_j^\vee \> = - \de_{ik}.
\end{gather*}
\end{lemma}

\begin{proof}
The two statements are equivalent by reversing the labeling of the simple roots, so it suffices to prove the first.  The claim is immediate if $k \geq i$.  For $k < i$, note that
\[
\< (c \om_i) - \om_i | \om_k^\vee + \sum_{j>k} C_{kj}\om_j^\vee \> = \< (s_k \cdots s_i \om_i) - \om_i | \om_k^\vee + \sum_{j>k} C_{kj}\om_j^\vee \>.
\]
A simple induction yields
\[
s_k \cdots s_i \om_i = \om_i + \sum_{j = k}^i\left( \sum_{a_1=j<\dots<a_{\ell}=i}(-1)^{\ell} \prod_{m = 1}^{\ell-1} C_{a_m,a_{m+1}} \right) \al_j,
\]
where the sum is taken over increasing sequences of any length from $j$ to $i$, and the product is taken to equal $1$ when $\ell = 1$.  From this we compute that $\< (s_k \cdots s_i \om_i) - \om_i | \om_k^\vee + \sum_{j>k} C_{kj}\om_j^\vee \> $ is equal to 
\begin{gather*}
\left( \sum_{a_1=k<\dots<a_{\ell}=i}(-1)^{\ell} \prod_{m = 1}^{\ell-1} C_{a_m,a_{m+1}} \right) + \sum_{j = k+1}^i\left( \sum_{a_1=j<\dots<a_{\ell}=i}(-1)^{\ell} \prod_{m = 1}^{\ell-1} C_{a_m,a_{m+1}} \right) C_{kj},
\end{gather*}
which vanishes since the two sums cancel.
\end{proof}

\begin{example}
The simplest example is $SL_2$, where $c$ is the nonidentity element of $W$ and $\mb{i} = (-1,1)$, $\mb{i}' = (1,-1)$ are the only double reduced words for $(c,c)$.  Their respective cluster variables are just matrix entries:
\[
(A_{-1},A_1,A_2) = (\De_{12},\De_{22},\De_{21}), \quad (A'_{-1},A'_1,A'_2) = (\De_{12},\De_{11},\De_{21}).
\]
The parametrization associated with $\mb{i}$ is
\[
x_{\mb{i}} \colon (X_{-1},X_1,X_2) \mapsto (X_{-1}X_1X_2)^{-\frac12}\begin{pmatrix} X_{-1}X_1X_2 & X_{-1}X_1 \\ X_1X_2 & 1+X_1 \end{pmatrix}.
\]
From this we can directly evaluate the matrix $N$ of \cref{lem:untwistedeval}, and along with the matrix $B^{\mathrm{mod}}_{\Si}$ we have
\[
N =\frac12 \begin{pmatrix} 1 & 1 & -1 \\ 1 & 1 & 1 \\ -1 & 1 & 1 \end{pmatrix}, \quad B^{\mathrm{mod}}_{\Si} = \begin{pmatrix} 1 & 1 & 0 \\ -1 & 0 & -1 \\ 0 & 1 & 1 \end{pmatrix}.
\]
From this we compute the matrix $M$ of \cref{prop:twistasclustertrans}, and the matrix $M'$ of \cref{eq:M'}:
\[
M = \begin{pmatrix} 0 & 0 & -1 \\ 0 & 1 & 0 \\ -1 & 0 & 0 \end{pmatrix}, \quad M' = \begin{pmatrix} 0 & 0 & -1 \\ -1 & 1 & -1 \\ -1 & 0 & 0 \end{pmatrix}
\]
\cref{prop:twistasclustertrans} then says that the twisted cluster variables are determined from these by
\begin{equation}\label{eq:twistedminors}
A'_i = \prod_{j \in I} (\tau^* A_j)^{M_{ij}}, \quad A_i = \prod_{j \in I} (\tau^* A'_j)^{M'_{ij}}.
\end{equation}

On the other hand, by expanding \cref{eq:twist} we compute the following explicit formula for the twist:
\[
\tau \colon \begin{pmatrix} a & b \\ c & d \end{pmatrix} \mapsto \begin{pmatrix} db^{-1}c^{-1} & b^{-1} \\ c^{-1} & d \end{pmatrix}.
\]
From this we can compute the twisted cluster variables directly:
\[
(\tau^*A_{-1},\tau^*A_1,\tau^*A_2) = (\De_{21}^{-1},\De_{11},\De_{12}^{-1}), \quad (\tau^*A'_{-1},\tau^*A'_1,\tau^*A'_2) = (\De_{21}^{-1},\De_{12}^{-1}\De_{22}\De_{21}^{-1},\De_{12}^{-1}).
\]
Of course, this agrees with \cref{eq:twistedminors}, noting that $M$ and $M'$ are each their own inverses.  
\end{example}

\bibliography{QSFDv1Arxiv}
\bibliographystyle{alpha}

\end{document}